[1994/06/01]

\documentclass{amsart}
\usepackage{amsthm}
\usepackage{amsmath,amsxtra,amssymb,latexsym, amscd,amsthm}
\usepackage[all]{xy}
\usepackage{mathrsfs}
\usepackage{enumerate}
\newtheorem{thm}{Theorem}[section]
\newtheorem{cor}[thm]{Corollary}
\newtheorem{prop}[thm]{Proposition}
\newtheorem{lem}[thm]{Lemma}
\theoremstyle{definition}

\newtheorem{exam}{Example}[section]  

\theoremstyle{remark}
\newtheorem{rmk}{Remark}[section]  
\newtheorem{\theequation}{}[section]

\renewcommand{\theequation}{\thesection.\arabic{equation}}
\usepackage{amsfonts}

\newcommand{\C}{\field{C}}

\def\2ovlOB{\overline {\overline \Omega}_B}
\def\1ovlO{\overline \Omega}

\def\fun{function}

\def\dlim{\displaystyle\lim}

\def\C{{\Bbb C}}

\def\P{{\Bbb P}}

\def\N{{\Bbb N}}

\begin{document}

\title[FR\'{E}CHET-VALUED FORMAL POWER SERIES]{FR\'{E}CHET-VALUED   FORMAL POWER SERIES}

\author{Thai Thuan Quang} 
\address{Department of Mathematics, 
Quy Nhon University,
170 An Duong Vuong, Quy Nhon, Binh Dinh, Vietnam.}
\email{thaithuanquang@qnu.edu.vn} 


%
%
\date{January 10, 2019}
\thanks{The research was supported by the National Foundation for Science and Technology Development (NAFOSTED), Vietnam, 101.02-2017.304}
\maketitle

\bigskip
\begin{abstract} Let $A$ be a non-projectively-pluripolar set in a Fr\'{e}chet space $E.$   We give  sufficient conditions to ensure the convergence  on some zero-neighbourhood in  $E$   of a (sequence of) formal power series of Fr\'{e}chet-valued continuous homogeneous polynomials provided that the convergence holds at a zero-neighbourhood of each complex line $\ell_a := \C a$ for every $a \in A.$
\end{abstract}

\bigskip
\section{Introduction}
In mathematics,  a formal power series is a  generalization of  polynomials as a formal object, where the number of terms is allowed to be infinite. 
The theory of formal power series   has drawn attention of mathematicians working in different branches because of their various applications. One can find applications of formal power series in classical mathematical analysis  and in the theory of Riordan algebras. Specially, this theory    lays the foundation for substantial parts of combinatorics and real and complex analysis. 

A formal power series $f(z_1, \ldots, z_n) = \sum c_{\alpha_1, \ldots, \alpha_n}z_1^{\alpha_1}\ldots z_n^{\alpha_n}$ in $\C^n,$ $n \ge 2,$ with coefficients in  $\C$ is said to be convergent if it converges absolutely in a zero-neighborhood  in $\C^n.$ A classical result of Hartogs   states that a series $f$ converges if and only if  $f_z(t) = f(tz_1, \ldots, tz_n)$ 
converges, as a series in $t,$ for all   $z \in \C^n.$ This can be interpreted as a formal analog
of Hartogs’ theorem on separate analyticity. Because a divergent power series still
may converge in certain directions, it is natural and desirable to consider the set of all
$z \in \C^n$ for which  $f _z$ converges. Since  $f_z (t)$ converges if and only if  $f_w (t)$ converges for all  $w \in \C^n$ on the affine line through $z,$  ignoring the trivial case  $z = 0,$ the set of directions along which $f$ converges can be identified with a subset of the projective space $\C\P^{n-1}.$  The convergence set  $\text{Conv}(f)$ of a divergent power series $f$ is defined to be the set of all directions $\xi \in \C\P^{n-1}$  such that  $f_z (t)$ is convergent for some  $z \in \varrho^{-1}(\xi)$ where $\varrho: \C^n \setminus \{0\} \to \C\P^{n-1}$  is the natural projection. In the two-variables case, Lelong \cite{Le} proved that $\text{Conv}(f)$ is an  $F_\sigma$-polar set (i.e. a countable union of closed sets of vanishing logarithmic capacity) in  $\C\P^1,$  and moreover, every  $F_\sigma$-polar subset of  $\C\P^1$ is contained in the $\text{Conv}(g)$ of some formal power series $g .$ The optimal result was later obtained by Sathaye \cite{Sa} who showed that the class of convergence sets of divergent power series in two-variables    is precisely the class of   $F_\sigma$-polar sets in  $\C\P^1.$ Levenberg  and Molzon, in \cite{LM},   showed that if the restriction of  $f $ on sufficiently many (non-pluripolar) sets of complex line passing through the origin is convergent on small neighborhood of $0  \in \C$ then $f$ is actually represent a holomorphic function near $0 \in \C^n.$ By using   delicate estimates on volume of complex varieties in projective spaces, Alexander \cite[Theorem 6.1]{Al}  showed that if the restriction of a sequence $(f_m)_{m \ge 1}$ of formal power series on every complex line passing through the origin in $\C^n$ is  convergent on compact sets (of the unit disk $\Delta \subset \C$) then     {$(f_m)_{m \ge 1}$ is the series of holomorphic function on the unit ball $\Delta_n \in \C^n$ which is convergent uniformly on compact sets.  By considering the class $PSH_\omega(\P^n)$ of  $\omega$-plurisubharmonic functions on  $\P^n$ with respect  to the form $\omega := dd^c\log|Z|$ on $\P^n,$   Ma   and   Neelon proved that a countable union of closed complete pluripolar sets in  $\P^n$ belongs to $\text{Conv}(\P^n).$ This generalizes the results of Lelong \cite{Le}, Levenberg and Molzon \cite{LM}, and Sathaye \cite{Sa}. In the same work, they also showed that each convergence set (of divergent power series) is a countable union of projective hulls of compact pluripolar sets. Recently, based on an investigation on a projectively pluripolar subset of $\C^n$ (via logarithmically homogeneous  plurisubharmonic function) Long and Hung \cite{LH} have shown that   a sequence $(f_m)_{m \ge 1}$ of formal power series in $\C^n$ converges uniformly on compact subsets of the ball $\Delta_n(r_0) \subset \C^n$ for some $r_0 > 0$ if for each $m \ge 1,$ the restriction of  $f_m$ to the complex line $\ell_a := \C a$ is holomorphic on the disk $\Delta(r_0) \subset \C$ for every $a \in A, $  a non-projectively-pluripolar set in $\C^n.$

 The main goal of this paper is to study the convergence of a  (sequence of)  formal power series of Fr\'{e}chet-valued continuous homogeneous polynomials. To prepear for the proofs of the main results, with the help of  techniques of pluripotential theory, we investigate  the Hartogs lemma for sequence of plurisubharmonic functions for the infinite dimensional case (Theorem \ref{thm_2}).  The first main result, Theorem \ref{thm_6}, gives a condition on a non-projectively-pluripolar set $A$ in a Fr\'{e}chet space $E$ such that a formal power series $f$ of  Fr\'{e}chet-valued continuous homogeneous polynomials of degree $n$ on $E$   converges in a neighbourhood of $0 \in E$ provided that  the restriction of $f$ on the complex line $\ell_a$ is convergent for every $a \in A.$ Theorem \ref{thm_2} also allows us to treat the problem on the extension to a entire function from the unit ball $\Delta_n \subset \C^n$ of a $C^\infty$-smooth function $f.$ Naturally, the condition ``\textit{for a non-projectively-pluripolar set $A,$ the restriction of $f$ on every $\ell_a,$ $a \in A,$  is  entire function}'' will be considered here (Theorem \ref{thm_7}). This result may be considered as a Fr\'{e}chet-valued version of the Forelli   theorem \cite{Sh}.

The problem considered in the last result, Theorem \ref{thm_8}, is giving  the conditions under which a sequence  of formal power series of Fr\'{e}chet-valued continuous homogeneous polynomials on $\C^n$  converges on a zero-neighbourhood. 
Another expression of this theorem will show that this is an extension of Alexander's theorem for the Fr\'{e}chet-valued case.

\section{Preliminaries}  
The standard notation of the theory of locally convex spaces   used in this note    is presented as  in the book of Jarchow \cite{Ja}. 
A locally convex space is always a complex vector space with a locally convex Hausdorff topology.
For a locally convex  space $E$ we use  $E'_{\rm bor}$ to denote  $E'$ equipped with the bornological topology associated with the strong topology $\beta.$ 

The locally convex structure of a Fr\'{e}chet space  is always assumed to be    generated by an increasing system $(\|\cdot\|_k)_{k \ge 1}$ of seminorms. For an absolutely convex subset $B$ of $E,$  by $E_B$ we denote the linear hull of $B$ which becomes a normed space in a canonical way if $B$ is bounded (with the norm $\|\cdot\|_B$ is the gauge \fun al of $B$).

Let  $D$ be a domain in a locally convex space $E.$  An  upper-semicontinuous function  $\varphi: D \to [-\infty, +\infty)$  is said to be \textit{plurisubharmonic}, and write $\varphi \in PSH(D),$ if $\varphi$ is subharmonic on every one dimensional section of $D.$

A subset $B \subset D$ 
is said to be   \textit{pluripolar} in $D$ if there exists  $\varphi \in PSH(D)$ such that $\varphi \not\equiv -\infty$ and $\varphi\big|_B = -\infty.$

A function $\varphi \in PSH(E)$ is called \textit{homogeneous plurisubharmonic} if
$$ \varphi(\lambda z) = \log|\lambda| + \varphi(z)\quad \forall \lambda \in \C, \ \forall z \in E . $$
We denote by $HPSH(E)$ the set of homogeneous plurisubharmonic functions on $E.$ 
We say that a subset $A \subset E$ is \textit{projectively pluripolar} if $A$ is contained in the $-\infty$ locus of some element in $HPSH(E)$ which is not identically $-\infty.$
It is clear that projective pluripolarity implies pluripolarity. The converse is not true (see \cite[Proposition 3.2 b]{LH}).

  Some properties,  examples and counterexamples of projectively pluripolar sets may be found in \cite{LH}.  We introduce below a few examples    in locally convex  spaces.
\begin{exam}\label{ex_1} Let $E$ be a metrizable locally convex space. Fix $a \in E.$ Then, the complex line $\ell_a,$ hence, every $A \subset \ell_a,$  is  projectively pluripolar in $E.$

Indeed, let $d$ be the metric defining the topology on $E.$ Consider the function
$$ \varphi(z) = -\log d(z, \ell_a) := -\log\inf_{w \in \ell_a}d(z, w).$$
 It  is easy to check that $\varphi \in HPSH(E),$ $\varphi \not\equiv -\infty$ and $\ell_a \subset \varphi^{-1}(-\infty).$
\end{exam}
\begin{exam} Let $E$ be a Fr\'{e}chet space which contains a non-pluripolar compact balanced convex subset $B.$ By the same proof as in Example \ref{ex_1}, the set $\partial B$ is pluripolar. However, $\partial B$ is not projectively pluripolar in $E.$

Otherwise, we can find a function $\varphi \in HPSH(E),$ $\varphi \not\equiv -\infty$ and $\partial B \subset \varphi^{-1}(-\infty).$ For every $z \in B$ we can write $z = \lambda y$ for some $y \in \partial B$ and $|\lambda| < 1.$ Then 
$$ \varphi(z) = \varphi(\lambda y) = \log|\lambda| + \varphi(y) = -\infty\quad\forall z \in B. $$
It is impossibe because $B$ is non-pluripolar.
\end{exam}
Hence, by Theorem 9 of \cite{DMV}, a nuclear Fr\'{e}chet space having  the linear topological invariant $(\widetilde\Omega)$ which is  introduced  by Vogt (see \cite{Vo}) contains a non-projectively-pluripolar set.


We use throughout this paper the following notations:
$$ \Delta_n(r) = \{z \in \C^n:\ \|z\| < r\}; \quad \Delta_n := \Delta_n(1);\quad 
\Delta(r) = \Delta_1(r);\quad \Delta := \Delta_1;  $$
and $\ell_a$ is the complex line $\C a.$ For further terminology from complex analysis we refer to \cite{Di}.
\section{The results}

First we investigate the Hartogs lemma for sequence of plurisubharmonic functions for the infinite dimensional case. This   is essential to our proofs.
\begin{lem}\label{lem_1} Let   $(P_n)_{n \ge 1}$ be a sequence of continuous homogeneous polynomials on   a Baire locally convex space $E$ of degree $\le n.$ Assume that
$$ \limsup_{n \to \infty}\frac 1n\log|P_n(z)|  \le M  $$
for each $z \in E$  and  some  constant  $M.$ Then for every $\varepsilon > 0$ and every compact set $K$ in $E$ there exists $n_0$ such that
$$ \frac 1n\log|P_n(z)| < M + \varepsilon \quad \forall n > n_0, \ \forall z \in K. $$
\end{lem}
\begin{proof}Without  loss  of  generality  we  may  suppose  that $M \le 0.$ Since
$$ \limsup_{n \to \infty}|P_n(z)|^{\frac 1n} \le 1 \quad \forall z \in E $$
the formular
$$ f(z)(\lambda) = \sum_{n \ge 1}P_n(z)\lambda^n $$
defines a function $f: E \to H(\Delta),$ the Fr\'{e}chet space of holomorphic functions on the open unit disc $\Delta \subset \C.$ 

Let us check $f$ is  holomorphic on $E.$ Given $z \in E \setminus \{0\}$ and consider $f(\cdot\ \! z): \C \to H(\Delta)$ with
$$  f(\xi z)(\lambda) = \sum_{n \ge 1}P_n(z)\lambda^n \xi^{k_n} $$
where $k_n = \deg P_n \le n.$ Then $f(\cdot\ \! z)$ is holomorphic because for $0 < r < 1$ we have
$$\aligned
 \lim_{n \to\infty}\sup_{|\lambda| \le r}(|P_n(z)||\lambda|^n)^{\frac 1{k_n}} &= \limsup_{n \to 0}(|P_n(z)|^{\frac 1n}r)^{\frac n{k_n}}\\
&\le \limsup_{n \to 0}|P_n(z)|^{\frac 1n}r \le r < 1.
\endaligned $$
This means $f$ is G\^ateaux holomorphic on $E.$

Now for each $k \ge 1$ we put
$$ A_k := \{z \in E:\ |P_n(z)| \le k^{k_n} \quad \forall n \ge 1\}. $$
By the continuity of $P_n,$ the sets  $A_k$ are closed in $E.$ Moreover, $E =  \bigcup_{k \ge 1}A_k.$ Since $E$ is a Baire space, there exists $k_0 \ge 1$ such that $\text{Int} A_{k_0} \neq \varnothing.$ Then $f$ is holomorphic on $\frac 1{k_0}A_{k_0}$ because
$$ \sum_{n \ge 1}|P_n(z)||\lambda|^n \le \sum_{n \ge 1}\frac{k_0^{k_n}}{k_0^{k_n}}r^n = \sum_{n \ge 1}r^n < \infty \quad \text{for} \ 0 < r < 1. $$
Hence, by Zorn's Theorem \cite[Theorem 1.3.1]{No}, $f$ is holomorphic on $E.$

Now given $K \subset E$ a compact set and $\varepsilon > 0.$ Take  $0 < r < 1$ and denote
$$ C := \sup\{|f(z)(\lambda)|:\ z \in K, |\lambda| \le r\} < \infty. $$
  Then we have 
$$ |P_n(z)| = \left|\frac 1{2\pi i}\int_{|\xi| = r}\frac{f(z)(\xi)d\xi}{\xi^{n + 1}}\right| \le \frac C{r^n} \quad \forall z \in K, $$
i.e,
$$ |P_n(z)|^{\frac 1n} \le \frac {C^{\frac 1n}}{r}.$$
Choose $n_0$ sufficiently we obtain
$$ |P_n(z)|^{\frac 1n} \le \frac {C^{\frac 1n}}{r} < e^\varepsilon \quad \forall n > n_0. $$
The lemma is proved.  
\end{proof}

The Proposition 5.2.1 in \cite{Kl} says that  a non-empty family $(u_\alpha)_{\alpha \in I}$  of  plurisubharmonic  functions from  the Lelong class such that the set $\{z \in \C^n:  \sup_{\alpha \in I}u_\alpha(z) < \infty\}$ is not $\mathcal L$-polar is locally uniformly bounded from above.

The next   is similar to the above result in the infinite dimensional case. 
\begin{thm}\label{thm_2} Let $B$ be a balanced convex compact subset of a Fr\'{e}chet space $E$ and $(P_n)_{n \ge 1}$ be a sequence of continuous homogeneous polynomials on $E$ of degree $\le n.$  Assume that the set
$$ \Big\{z \in E_B: \sup_{n \ge 1}\frac 1n\log|P_n(z)| < \infty\Big\} $$
is not projectively pluripolar in $E_B.$ Then the family $(\frac 1n\log|P_n|)_{n \ge 1}$ is locally uniformly bounded from above on $E_B.$
\end{thm}
 \begin{proof} Suppose that the family $(\frac 1n\log|P_n|)_{n \ge 1}$ is not locally uniformly bounded from above on $B.$ Then there exists a sequence $(u_j)_{j \ge 1} = (\frac 1{n_j}\log|P_{n_j}|)_{j \ge 1}$ such that
$$ M_j := \sup_{z \in B}u_j(z) \ge j \quad \forall j \ge 1. $$
Take $w \in E_B \setminus B$ and for each $j \ge 1$ consider the function
$$ v_j(\zeta) := u_j (\zeta^{-1}w) - M_j - \log^+(|\zeta|^{-1}\|w\|_B), \quad \text{for}\  \zeta \in \Delta(\|w\|_B) \setminus \{0\}. $$
Obviously,  $v_j$  is subharmonic and, it is easy to see that $v_j(\zeta) \le O(1)$ as $\zeta \to 0.$  Hence, in view of Theorem 2.7.1 in \cite{Kl},  $v_j$ extends  to  a subharmonic function, say $\widetilde v_j,$ on $\Delta(\|w\|_B).$ Now, by  the  maximum principle,  $\widetilde v_j  \le 0$ on $\Delta(\|w\|_B).$  In  particular,
$$  v_j(1) =  \widetilde v_j(1) = u_j (w) - M_j - \log^+\|w\|_B \le 0. $$
Hence
\begin{equation}\label{eq_1} 
u_j(z) - M_j \le \log^+\|z\|_B \quad \text{for} \ z \in E_B, \forall j \ge 1. 
\end{equation}
Then there exists $z_0 \in E_B$ such that
\begin{equation}\label{eq_2}
\limsup_{j \to \infty}\exp{(u_j(z_0) - M_j)} =: \delta > 0. 
\end{equation}
For otherwise  we  would have $\limsup_{j \to \infty}\exp(u_j(z) - M_j) \le 0$   at  each point $z \in E_B.$ Note that, by Lemma \ref{eq_1}, the sequence $(u_j(z) - M_j)_{j \ge 1},$ hence, $(\exp(u_j(z) - M_j))_{j \ge 1}$ is bounded from above on any compact set in $E_B.$ This would imply from \cite[Lemma 1.1.12]{No}  that $\exp{(u_j(z) - M_j)} < \frac 12$ for all $z \in B$  and  all sufficiently large $j.$ But  then  the  last estimate  would contradict  the  definition  of  the  constants $M_j.$

Now we choose a subsequence $(u_{j_k})_{k \ge 1} \subset (u_j)_{j \ge 1}$ such that
$$ \lim_{k \to \infty}\exp(u_{j_k}(z_0) - M_{j_k}) = \delta \quad \text{and}\quad M_{j_k} \ge 2^{k} $$
for all $k \ge 1.$ Consider the function
$$ w(z) := \sum_{k \ge 1}2^{-k}(u_{j_k} - M_{j_k}), \quad z \in E_B. $$
In view of (\ref{eq_1}) we have the extimate
$$ w_k(z) := 2^{-k}(u_{j_k} - M_{j_k}) - 2^{-k}\log R \le 0 $$
for $z \in E_B,$ $\|z\|_B \le R$ and $R \ge 1.$ Thus $w_k$ is plurisubharmonic on $\{z \in E_B: \ \|z\|_B < R\}$ and $w_k \le 0.$ Hence, the function $\sum_{k \ge 1}w_k = w - \log R,$ $R > 1,$ is either plurisubharmonic on  $\{z \in E_B: \ \|z\|_B < R\}$  or identically $-\infty.$ Consequently, as  $R$  can  be chosen  arbitrarily  large,  $w$  is  either plurisubharmonic or identically $-\infty.$ Therefore, since  $w(z_0) > -\infty,$ $w \in PSH(E_B).$ It is easy to see that $w \in HPSH(E_B).$ 

If $z \in E_B,$ $\sup_{n \ge 1}\frac 1n\log|P_n(z)| < \infty$ then $\sum_{k \ge 1}2^{-k}u_{j_k}(z) < \infty$ and, hence
$$ w(z) \le \sum_{k \ge 1}2^{-k}u_{j_k}(z) - \sum_{k \ge 1}1 = -\infty $$
which proves  that  the  set   
$$ \Big\{z \in E_B: \sup_{n \ge 1}\frac 1n\log|P_n(z)| < \infty\Big\} $$
is projectively pluripolar in $E_B.$ This contradics the hypothesis.
\end{proof}

\begin{cor}\label{cor_3} Let $B, E$ and $(P_n)_{n \ge 1}$ be as in   Theorem \ref{thm_2}; in addition assume that $B$ contains a non-projectively-pluripolar subset. Then the family $(\frac 1n\log|P_n|)_{n \ge 1}$ is locally uniformly bounded from above on $E.$
\end{cor}
\begin{proof} It suffices to prove that $E_B$ is dense in $E.$ Indeed, if the closure of the   subspace $E_B$  is not equal to $E$ then, by the Hahn-Banach theorem, there exists $\varphi \in E',$ $\varphi \neq 0,$ such that $\varphi(E_B) = 0.$ Then it is easy to see that $v := \log|\varphi| \in HPSH(E),$  $v \not\equiv 0,$ $B \subset E_B \subset \{z: \ v(z) = -\infty\}.$ This contradicts the fact that $B$ contains a non-projectively-pluripolar subset.
\end{proof}
It is known that a subset with non-empty interior in a Fr\'{e}chet space is not pluripolar, hence it is not projectively pluripolar. Then by Corollary \ref{cor_3} we have the following.
\begin{cor}\label{cor_4} Let $B$ be a balanced convex compact subset of a Fr\'{e}chet space $E$ which contains a non-projectively-pluripolar subset  and $(P_n)_{n \ge 1}$ be a sequence of continuous homogeneous polynomials on $E$ of degree $\le n.$  If the set
$$ \Big\{z \in E_B: \sup_{n \ge 1}\frac 1n\log|P_n(z)| < \infty\Big\} $$
has the non-empty interior in $E$ then the family $(\frac 1n\log|P_n|)_{n \ge 1}$ is locally uniformly bounded from above on $E.$
\end{cor}  
\begin{lem}\label{lem_5} A regular inductive limit $E = \displaystyle\varinjlim_{n \to \infty}E_n$ of 
countable family of locally convex spaces satisfies the countable boundedness condition (c.b.c), i.e. if given a sequence $\{B_n\}$ of bounded subsets of $E,$ there are $\lambda_n > 0,$ $n \ge 1,$ such that $\bigcup_{n \ge 1}\lambda_nB_n$ is bounded.
\end{lem}
 \begin{proof} 
Given a sequence $(B_n)_{n \ge 1}$ of bounded subsets of $E.$ By the regularity of $E,$ for each $n \ge 1,$ there exists $k_n \ge 1$ such that $B_n$ is bounded in $E_{k_n}.$ Without loss of generality we may assume that $k_n = n.$ Hence we can find a sequence of positive numbers $(\lambda_n)_{n \ge 1}$ such that $\lambda_n B_n \subset U_n,$ a zero-neighbourhood in $E_n,$ for all $n \ge 1.$ Obviously, $\bigcup_{n \ge 1}\lambda_nB_n \subset U := \bigcup_{n \ge 1}U_n,$ a zero-neighbourhood in $E.$ Lemma is proved.
\end{proof}
\begin{thm}\label{thm_6}  Let $A$ be a non-projectively-pluripolar set which is contained in a   balanced convex compact subset of a Fr\'{e}chet space $E$   and $f = \sum_{n \ge 1}P_n$ be a formal power series where $P_n$ are continuous homogeneous polynomials of degree $n$ on $E$ with values in a Fr\'{e}chet space $F.$  If for each $a \in A,$ the restriction of $f$ on the complex line $\ell_a$ is convergent then $f$ is convergent in a neighbourhood of $0 \in E.$
\end{thm}
\begin{proof} We divide the proof into two steps: 

(i) \textit{Step 1: We consider the case where $F = \C.$ } It follows from the hypothesis that
$$ \limsup_{n \to \infty}|P_n(z)|^{\frac 1n} < \infty \quad \forall z \in A. $$
Then, by Corollary \ref{cor_3} there exists a zero-neighbourhood $U$ in $E$ such that
$$ \sup\{|P_n(z)|^{\frac 1n}:\ z \in U, n \ge 1\} =: M < \infty. $$
This implies that $f$ is uniformly convergent on $(2M)^{-1}U.$

\vskip0.2cm
(ii) \textit{Step 1: We consider the case where $F$ is Fr\'{e}chet.} By the step 1 we can define the   linear map 
$$T: F'_{\rm bor} \to H(0_E)$$ 
by letting
$$ T(u) = \sum_{n \ge 1}u(P_n) $$
where  $H(0_E)$ denotes the space of germs of scalar holomorphic functions at $0 \in E.$
Suppose that $u_\alpha \to u$ in $F'_{\rm bor}$ and $T(u_\alpha) \to v$ in $H(0_E) $ as $\alpha \to \infty.$  This implies, in particular, that $[T(u_\alpha)](z) \to v(z)$ for all $z$ in some zero-neighbourhood $U$ in $E.$  However, for $z \in U$ we have
$$\aligned
 \left[T(u_\alpha - u)\right](z) &= \sum_{n \ge 1}(u_\alpha - u)(P_n(z)) = \lim_{n \to \infty}\sum_{k = 1}^n(u_\alpha - u)(P_n(z))  \\
&= (u_\alpha - u)\Big(\lim_{n \to \infty}\sum_{k =1}^nP_n(z)\Big) = (u_\alpha - u)\Big(\sum_{n \ge 1}P_n(z)\Big).
\endaligned$$
Then $[T(u_\alpha)](z) \to [T(u)](z)$ for all $z \in U.$ This implies that $v = T(u).$ Hence $T$ has a closed graph. 
 
 Meanwhile, since $F$ is Fr\'{e}chet, by \cite[Theorem 13.4.2]{Ja} we have $\beta(F', F)_{\rm bor} = \eta(F', F)$ on $F'.$ This implies that   $F'_{\rm bor}$ is ultrabornological.
  On the other hand, because $E$ is metrizable,  we have
 $$ H(0_E) =  \varinjlim_{n \to \infty}(H^\infty(V_n), \|\cdot\|_n) $$
where $(V_n)_{n \ge }$ is a  countable fundamental  neighbourhood system at $0 \in E,$ and $\|\cdot\|_n$ is the norm on the Banach space $H^\infty(V_n)$ given by $\|f\|_n = \sup_{z \in  V_n}|f(z)|.$ 
Hence, by the closed graph theorem of Grothendieck \cite[Introduction, Theorem B]{Gr}, $T$ is continuous. 

Now, by \cite[Lemma 4.33]{Di} and Lemma \ref{lem_5}, $H(0_E)$   satisfies (c.b.c). Using Proposition 1.8 in \cite{BG} we deduce that there exists a neighbourhood $V$ of $0 \in E$ such that $T: F'_{\rm bor} \to H^\infty(V)$ is continuous linear. 

 Now we  define the map $\widehat f: V \to F''_{\rm bor}$ by the formula
$$[\widehat{f}(z)](u) = [T(u)](z), \quad z \in  V, \; u \in F'_{\rm bor}.$$
Since $T$ is continuous and point evaluations on $H(V)_{\rm bor}$ (see \cite[Proposition 3.19]{Di}) are continuous it follows that $\widehat f(z) \in F''_{\rm bor}$ for all $z \in V.$  Moreover, for each fixed $u \in F'_{\rm bor}$ the mapping
$$ z \in V \mapsto  [T(u)](z)$$
is holomorphic, that is
$$ \widehat f: V \to (F''_{\rm bor}, \sigma(F''_{\rm bor}, F'_{\rm bor})) $$
is a continuous mapping. For all $a \in V, b \in E$ and all $u \in F'_{\rm bor}$ the mapping
$$ \{t \in \C:\ a + t b \in V\} \ni \lambda  \mapsto u \circ \widehat f(a + \lambda b)$$ 
is a G\^ateaux holomorphic mapping and hence
$$ \widehat f: V \to (F''_{\rm bor}, \sigma(F''_{\rm bor}, F'_{\rm bor})) $$
is holomorphic.

By \cite[8.13.2 and 8.13.3]{Ja}, $F'_{\rm bor}$ is a complete locally convex space. Hence
by \cite[Theorem 4, p.210]{Ho}  applied to the complete space $F'_{\rm bor}$ we see that
$(F''_{\rm bor}, \sigma(F''_{\rm bor}, F'_{\rm bor}))$ and $(F'_{\rm bor})'_\beta$ have the same bounded sets. An application of \cite[Proposition 13]{Na}  shows that
$$ \widehat f: V \to (F'_{\rm bor})'_\beta $$
is holomorphic.

Let $j$ denote the canonical injection from $F$ into $F''.$ If $z \in B := V \cap \{ta:\ t \in \C, a \in A\}$ and $\widehat f(z) \neq j(f(z))$ then there exists $u \in F'$ such that  
$$ \widehat f(z)(u) \neq j(f(z))(u) = u(f(z)).  $$
This, however, contradicts the fact that for all $z \in B$ we have
$$ \widehat{f}(z)(u) = [T(u)](z) = \sum_{n \ge 1}u(P_n)(z) = u(f(z)). $$
We now fix a non-zero  $z \in B.$  Then there exists a unique sequence in $F'',$ $(a_{n, z})_{n=1}^\infty,$ such that for all $\lambda \in \C$
$$ \widehat f (\lambda z) = \sum_{n =0}^\infty a_{n, z}\lambda^n. $$
Since $\widehat f(0) = f(0) = a_{0, z}$ it follows that $a_{0, z} \in F.$  Now suppose that $(a_{j, z})_{j = 0}^n \subset F.$ When $|\lambda| \le 1,$ $\widehat f(\lambda z) = f(\lambda z) \in F.$ Hence, if $\lambda \in \C,$ $0 < |\lambda| < 1,$ then
$$ \frac{\widehat f(\lambda z) - \sum_{j=0}^na_{j, z}\lambda^j}{\lambda^{n+1}} = \sum_{j=n+1}^\infty a_{j, z}\lambda^{j - n -1} \in F. $$
Since $F$ is complete we see, on letting $\lambda$ tend to $0,$ that $a_{n+1, z} \in F.$ By induction $a_{n, z} \in F $ for all $n$ and hence $\widehat f(\lambda z) \in F$ for all $\lambda \in \C$ and all $z \in B.$ Since $\widehat f$ is continuous and $F$ is a closed subspace of  $(F'_{\rm bor})'_\beta$ (see \cite[Lemma 2.1]{QVHB}) we have shown that 
$\widehat f : V \to F$ is holomorphic.
 
Hence, the series $\sum_{n \ge 1}P_n$ is convergent on $V$ to $f.$
\end{proof}
\begin{thm}\label{thm_7} Let $F$ be a Fr\'{e}chet space, $f: \Delta_n \to F$ be a function which belongs to $C^k$-class at $0 \in \C^n$ for $k \ge 0$ and $A \subset \C^n$ be a non-projectively-pluripolar set. If the restriction of $f$ on each complex line $\ell_a,$ $a \in  A,$ is holomorphic then   there exists an entire function $\widehat f$ on $\C^n$ such that $\widehat f = f$ on $\ell_a$ for all $a \in A.$
\end{thm}
\begin{proof} By the hypothesis, for each $k \ge 0$ there exists $r_k \in (0, 1)$ such that $f$ is a $C^k$-function on $\Delta_n(r_k).$ We may assume that $r_k \searrow 0.$ Put
$$ P_k(z) = \frac 1{2\pi i}\int_{|\lambda| = 1}\frac{f(\lambda z)d\lambda}{\lambda^{k + 1}}, \quad z \in \Delta_n(r_k). $$
Then, for each $k \ge 0$ and $p \ge k,$ $P_m$ is a bounded $C^p$-function on $\Delta_n(r_p).$ Since $\lambda \mapsto f(\lambda a)$ is holomorphic for all $a \in A$ we deduce that
\begin{equation}\label{eq_3} 
P_k(\lambda a) = \lambda^kP_k(a) \quad \text{for}\ a \in  A, \lambda \in \C.
\end{equation}
 By the boundedness of $P_k$ on $\Delta_n(r_k)$ we have
$$ P_k(w) = O(|w|^k) \quad \text{as}\ w \to 0. $$
On the other hand, since $P_k \in C^{k + 1}(\Delta_n(r_{k + 1})),$ the Taylor expansion of $P_k$  at $0 \in \Delta_n(r_{k + 1})$ has the form
\begin{equation}\label{eq_5} P_k(z) = \sum_{\alpha + \beta = k}P_{k, \alpha, \beta}(z) + |z|^k\varrho(z) 
\end{equation}
where $P_{k, \alpha, \beta}$  is a polynomial of degree $\alpha$ in $z$ and degree $\beta$ in $\overline z$ and $\varrho(z) \to 0$ as $z \to 0.$
  
In (\ref{eq_5}), replacing $z$ by $\lambda z,$ $|\lambda| < 1,$  from (\ref{eq_3}) we obtain
\begin{equation}\label{eq_6} 
\sum_{\alpha + \beta = k}P_{k, \alpha, \beta}(z)\lambda^\alpha\overline{\lambda}^\beta  + |\lambda|^k|z|^k\varrho(\lambda z) =   \sum_{\alpha + \beta = k}P_{k, \alpha, \beta}(z)\lambda^k + \lambda^k|z|^k\varrho(z) 
\end{equation}
for $z \in  r_{k+1} A.$
 
 This yields that $\varrho(\lambda  z) = \varrho(z)$ for $\lambda \in [0, 1),$ and hence, $\varrho(z) = \varrho(0) = 0$ for $z \in  r_{k+1} A.$ Thus
$$ P_{k, \alpha, \beta}(z) = 0 \quad\text{for}\ \beta > 0 \ \text{and}\ z \in  r_{k+1}A. $$
Note that $r_{k+1}A$  is also not projectively pluripolar. It is easy to check that 
$$ P_{k, \alpha, \beta} = 0 \quad\text{for}\ \beta > 0. $$
Indeed, for every $\varphi \in F',$   the function
$$ u(w) = \frac 1{\deg P_{k, \alpha, \beta}}\log|(\varphi \circ P_{k, \alpha, \beta})(w)| $$
is homogeneous plurisubharmonic on $\C^n,$ $u \equiv -\infty$ on $r_{k+1}A.$ Since $r_{k+1}A$ is not projectively pluripolar, it implies
that $u \equiv -\infty$ and hence $\varphi \circ P_{k, \alpha, \beta} \equiv 0$ on $\C^n$ for every $\varphi \in F'.$ It implies that  $ P_{k, \alpha, \beta} \equiv 0$ on $\C^n$ for $\beta > 0.$

Thus, from (\ref{eq_5}) we have
$$ P_k(z) = P_{k, k, 0}(z) = \sum_{|\alpha| = k}c_\alpha z^\alpha $$
for $z \in \Delta_n(r_{k+1})$ and $P_k$ is a homogeneous holomorphic polynomial of degree $k.$

Now, let $(\|\cdot\|_m)_{m \ge 1}$  be an increasing fundamental system of continuous semi-norms defining the topology of $F.$  By the hypothesis, for every $m \ge 1$ 
$$ \limsup_{k \to \infty}\frac 1k\log\|P_k(z)\|_{m} = -\infty \quad\text{for} \ z \in A. $$
Then, by Corollary \ref{cor_3}, the sequence $(\frac 1k\log\|P_k(z)\|_{m})_{k \ge 1}$ is locally uniformly bounded from above on $\C^n$ for all $m \ge 1.$ Thus we can define
$$ u_{m}(z) = \limsup_{k \to \infty}\frac 1k\log\|P_k(z)\|_{m}, \quad z \in \C^n. $$
By \cite{Si} the upper semicontinuous regularization $u_m^*$ of $u_m$ belongs to the Lelong class $\mathcal L(\C^n)$ of  plurisubharmonic functions with logarithmic growth on $\C^n.$ Moreover, by  Bedford-Taylor's theorem \cite{BT} 
$$ S_m := \{z \in \C^n:\ u_m^*(z) \neq u_m(z)\} $$
is pluripolar for all $m \ge 1.$

On the other hand, by \cite{LH}, $A^* := \{ta:\ t \in \C, a \in A\}$ is not pluripolar. This yields that $u_m^* \equiv -\infty$ for all $m \ge 1$ because $u^*_m = u_m = -\infty$ on $A^* \setminus S_m$ and $A^* \setminus S_m$ is non-pluripolar.  Since $u_m^* \ge u_m$ we have $u_m \equiv -\infty$ for $m \ge 1.$ Hence the series $\sum_{k \ge 0}P_k(z)$ is convergent for $z \in \C^n$ and it defines a holomorphic extension $\widehat f$ of $f\big|_{\ell_a}$ for every $a \in A.$
\end{proof}
\begin{thm}\label{thm_8} Let $A \subset \C^n$ be a non-projectively-pluripolar set  and $(f_\alpha)_{\alpha \ge 1}$   be a sequence of formal power series of continuous homogeneous polynomials on $\C^n$ with values in a Fr\'{e}chet space $F.$ Assume that there exists $r_0 \in (0, 1)$ such that, for each $a \in A,$  the restriction of  $(f_\alpha)_{\alpha \ge 1}$ on $\ell_a$ is a sequence of holomorphic functions which is convergent on the disk $\Delta(r_0).$   Then there exists $r > 0$ such that $(f_\alpha)_{\alpha \ge 1}$  is a sequence of holomorphic functions that converges on $\Delta_n(r).$
\end{thm}

The proof of this theorem requires some extra   results concerning to  Vitali's theorem for a sequence of Fr\'{e}chet-valued holomorphic functions.
\begin{rmk}\label{rmk1} In exactly the same way,   Theorem 2.1 in \cite{AN} is true for the Fr\'{e}chet-valued case.
\end{rmk}

\begin{lem}\label{lem_ext1} Let $E, F$ be  Fr\'echet spaces,  $D \subset E$ be an open set.  Let $f: D \to F$ be a locally bounded function such that $\varphi \circ f$ is holomorphic for all $\varphi \in W \subset F',$ where $W$ is separating. Then $f$ is holomorphic.
\end{lem}
The proof of Lemma runs as in the proof of Theorem 3.1 in \cite{AN}, but here we use Vitali's theorem in (\cite{BS}, Prop. 6.2) which states for a sequence of holomorphic functions on an open connected subset of a locally convex space.
\begin{lem} \label{lem2}  Let $D$ be a domain in a Fr\'echet  space $E$  and  $f: D \to F$ be holomorphic, where  $F$ is a barrelled locally convex space. Assume that  $D_0 = \{z \in D:\; f(z) \in G\}$ is not rare in $D,$  where $G$ is a closed subspace of $F.$ Then $f(z) \in G$ for all $z \in D.$
\end{lem}
\begin{proof} (i) We first consider the case $G = \{0\}.$ On the contrary, suppose that $f(z^*) \not= 0$ for some $z^* \in D \setminus D_0.$ By the Hahn-Banach theorem, we can find $\varphi \in F'$ such that  $(\varphi \circ f)(z^*) \not= 0.$ Let $z_0 \in (\text{int}\overline{D_0}) \cap D$ and let $W$ be a balanced convex neighbourhood of $0 \in E$ such that $z_0 + W \subset \overline{D_0}.$ Then by the continuity of $f$ we deduce that $f = 0$ on $z_0 + W.$ Hence, it follows from the identity theorem (see \cite{BS}, Prop. 6.6) that $f = 0$ on $D.$ 
This contradicts above our claim $(\varphi \circ f)(z^*) \not= 0.$

(ii) For the general case, consider the quotient space $F/G$ and the holomorphic function  $\omega \circ f: D \to F/G$ where $\omega: F \to F/G$ is the canonical map. Then $\omega \circ f \equiv  0$ on $D_0.$ By the case (i), $\omega \circ f \equiv 0$ on $D.$ This means that $f(z) \in G$ for all $z \in D.$ \end{proof}

\begin{prop}\label{prop_ext1} Let $E, F$ be Fr\'echet spaces and $D \subset E$ a domain. Assume that $(f_i)_{i \in \N}$ is a  locally bounded sequence of holomorphic functions on $D$ with values in $F.$ 
Then the following assertions are equivalent:
\begin{enumerate}
\item[\rm (i)] The sequence $(f_i)_{i \in \N}$ converges uniformly on all compact subsets of  $D$ to a holomorphic function $f: D \to F;$ 
\item[\rm (ii)]  The set $D_0 = \{z \in D: \dlim_if_i(z)\; \text{exists}\}$ is not rare in $D.$  
\end{enumerate}
\end{prop}
\begin{proof} It suffices to prove the implication (ii) $\Rightarrow$ (i) because the case (i) $\Rightarrow$ (ii) is trivial. 
%
%
Define $\widetilde{f}: D \to \ell^\infty(\N, F)$ by $\widetilde{f}(z) = (f_i(z))_{i \in \N},$ where $\ell^\infty(\N, F)$ is the Fr\'echet space with the topology induced by the system of semi-norms
$$|\!|\!|x|\!|\!|_k = |\!|\!|(x_\alpha)_{\alpha \in \N}|\!|\!|_k = \sup_{\alpha}\|x_\alpha\|_k,\; \forall k, \; \forall x = (x_\alpha) \in \ell^\infty(\N, F).$$

For each $k \in \N$ we denote
$pr_k: \ell^\infty(\N, F) \to F$ is the $k$-th projection with $pr_k((w_i)_{i \in \N}) = w_k.$ Obviously 
$$W = \{\varphi \circ pr_k; \;   \varphi \in F', \;   k \in \N\} \subset \ell^\infty(\N, F)'$$
is separating and 
$$\varphi \circ pr_k \circ \widetilde{f} = \varphi \circ pr_k \circ (f_i)_{i \in \N} = \varphi \circ f_k$$
is holomorphic for every $k \in \N.$ Then by Lemma \ref{lem_ext1}, $\widetilde{f}$ is holomorphic.

Since the space 
$$G = \{(w_i)_{i \in \N} \in \ell^\infty(\N, F): \; \displaystyle\lim_{i \to \infty}w_i \;\text{exists}\}$$
is closed, by the hypothesis, $\widetilde{f}(z) \in G$ for all $z \in D_0.$ It follows from Lemma \ref{lem2} that $\widetilde{f}(z) \in G$ for all $z \in D.$ Thus $f(z) = \lim_{i \to \infty}f_i(z)$ exists for all $z \in D.$ Note that $\Phi: G \to F$ given by $\Phi((y_i)_{i \in \N}) = \lim_{i \to \infty}y_i$ defines a bounded operator. Therefore $f = \Phi \circ \widetilde{f}$ is holomorphic.

Finally, in order to prove that $(f_i)_{i \in \N}$ converges uniformly on  compact sets in $D$ to $f,$ it suffices to show that $(f_i)_{i \in \N}$ is locally uniformly convergent in $D$ to $f.$  Since $(f_i)_{i \in \N}$ is locally bounded, by (\cite{BS}, Prop. 6.1) $(f_i)_{i \in \N}$ is equicontinuous at every $a \in D.$  Let $a$ be fixed point of $D.$ Then for every balanced convex neighbourhood $V$ of $0$ in $F$  there exists a neighbourhood $U_a^1$ of $a$ in $D$ such that
\begin{equation}\label{tag1}
f_i(z) - f_i(a) \in 3^{-1}V,\quad \forall z \in U_a^1, \; \forall i \in \N.\end{equation}
 Since $\dlim_{i \to \infty}f_i = f$ in $D,$ we can find $i_0 \in \N$ such that
\begin{equation}\label{tag2}
f_i(a) - f(a) \in  3^{-1}V,\quad \; \forall i \ge i_0.\end{equation}
By the continuity of $f,$ there exists a neighbourhood $U_a^2$ of $a$ in $D$ such that
\begin{equation}\label{tag3}
f(a) - f(z) \in  3^{-1}V,\quad \forall z \in U_a^2.\end{equation}
From (\ref{tag1}), (\ref{tag2}) and (\ref{tag3}), for all $z \in U_a = U_a^1 \cap U_a^2$ for all $i \ge i_0$ we have
\begin{equation}\label{tag4}
f_i(z) - f(z) \in V.\end{equation}

The proof of the proposition is complete. \end{proof}
 
We now can prove  Theorem \ref{thm_8} as follows.

\vskip0.2cm
\noindent
\textit{Proof of  Theorem \ref{thm_8}.} \ 
  As in the proof of Theorem \ref{thm_6}, for each $\alpha \ge 1,$ define the continuous linear map $T_\alpha: F'_{\rm bor} \to H(0_{\C^n})$ given by
$$ T_\alpha(u) = u \circ f_\alpha, \quad u \in F'_{\rm bor}. $$
By Theorem 3.5 in \cite{LH}, the sequence $(T_\alpha(u))_{\alpha \ge 1}$ converges in $H(0_{\C^n})$ for every $u \in F'_{\rm bor}.$ Since $F'_{\rm bor}$ is barrelled (see \cite[13.4.2]{Ja}) it follows that the sequence $(T_\alpha)_{\alpha \ge 1}$ is equicontinuous in $L(F'_{\rm bor}, H(0_{\C^n}))$ equipped with the strong topology. As in the proof of Theorem \ref{thm_6}, by Proposition 1.8 in \cite{BG} we deduce that there exists a neighbourhood $U$ of $0 \in F'_{\rm bor}$ such that 
$$ \bigcup_{\alpha \ge 1}T_\alpha(U) $$
is bounded in $H(0_{\C^n}).$ By the regularity of $H(0_{\C^n}),$ we can find $r \in (0, r_0)$ such that $\bigcup_{\alpha \ge 1}T_\alpha(U)$ is contained and bounded in $H^\infty(\Delta_n(r)).$ This yields that $(f_\alpha)_{\alpha \ge 1}$ is contained and bounded in $H^\infty(\Delta_n(r), F).$ Since for each $z \in \Delta_n(r)$ the sequence  $(f_{\alpha}\big|_{\ell_z})_{\alpha \ge 1}$ is convergent in $\Delta_1(r_0) \subset \ell_z,$ by Remark \ref{rmk1}, the sequence $(f_\alpha(z))_{\alpha \ge 1}$ is convergent for every $z \in \Delta_n(r).$ On the other hand, because $(f_{\alpha})_{\alpha \ge 1}$ is  bounded  in $H^\infty(\Delta_n(r), F),$ by Proposition \ref{prop_ext1} it follows that the sequence $(f_\alpha)_{\alpha \ge 1}$ is convergent in $H(\Delta_n(r), F).$
\hfill$\square $

\section{Discussion and  open question.}
  From Proposition 3.1 in \cite{LH} it is clear that, in $\C^n,$ the following are equivalent:
\begin{itemize}
\item[a)] $A$ is projectively pluripolar;
\item[b)] $A^\lambda := \{tz:\ t \in \C, |t| < \lambda, z \in A\}$ is  pluripolar for each $\lambda > 0;$
\item[c)] $\mu(A^\lambda) = 0$ where $\mu$ is the Lebesgue measure;
\item[d)] $\nu(\varrho(A^\lambda)) = 0$ where $\nu$ is the invariant measure  on the projective space $\C\P^{n - 1}$ and $\varrho: \C^n \setminus \{0\} \to \C\P^{n - 1}$ is the natural projection.
\end{itemize}
 
Thus, we can restate Theorem \ref{thm_8} in an alternative form as follows to obtain  an extension of Hartogs' result (cf. \cite[Corollary 6.3]{Al}) that is an immediate consequence Alexander's theorem   from the scalar case to the Fr\'{e}chet-valued one. 
\begin{thm} Let $(f_{\alpha})_{\alpha \ge 1}$ be a sequence of Fr\'{e}chet-valued holomorphic functions on $\Delta_n \subset \C^n.$  Let $B$ be a  subset of $\Delta_n$ such that $\nu(\varrho(B)) = 0$ where $\nu$ is the invariant measure  on the projective space $\C\P^{n - 1}$ and $\varrho: \C^n \setminus \{0\} \to \C\P^{n - 1}$ is the natural projection. Assume that for some $r_0 \in (0, 1),$ the restriction of the sequence $(f_\alpha)_{\alpha \ge 1}$ on each complex line $\ell$ through $0 \in \Delta_n$ with $\ell \cap B = \{0\}$ is convergent in $\Delta_1(r_0).$ Then  $(f_\alpha)_{\alpha \ge 1}$  is the sequence of holomorphic functions which converges on a zero-neighbourhood in $\Delta_n.$
\end{thm}

One question still unanswered is whether we can  obtain a truly generalization of Alexander's theorem (cf. \cite[Theorem 6.2]{Al}) for Frechet-valued version?
 In other word, whether Theorem \ref{thm_8}  is true or not if the uniform convergence of the family $(f_{\alpha})_{\alpha \ge 1}$ on the disk $\Delta_1(r_0) \subset \ell_a$  for each $a \in A$ is replaced by \textit{normality} of this family on  $\Delta_1(r_0)$?

\end{document}